\documentclass[11pt, twoside, reqno]{amsart}

\usepackage[usenames,dvipsnames,svgnames,table]{xcolor}
\usepackage[pagebackref,linktocpage=true,colorlinks=true,linkcolor=Blue,citecolor=Green,urlcolor=RoyalBlue]{hyperref}
\usepackage[alphabetic]{amsrefs}

\numberwithin{equation}{section}

\usepackage[colorinlistoftodos,prependcaption,textsize=tiny]{todonotes}

\usepackage[mathscr]{euscript} 

\usepackage{pxfonts}

\theoremstyle{plain}
\newtheorem{theorem}{Theorem}[section]
\newtheorem{definition}[theorem]{Definition}
\newtheorem{lemma}[theorem]{Lemma}
\newtheorem{prop}[theorem]{Proposition}
\newtheorem{remark}[theorem]{Remark}

\newcommand\R{\mathbb R} 
\renewcommand\S{\mathbb S} 
\renewcommand\H{\mathcal H}
\newcommand{\capacity}{\mathrm{cap}}

\newcommand{\supp}{{ \rm supp}}
\newcommand{\dist}{{ \rm dist}}
\newcommand\p{\partial}
\newcommand{\gstand}{ \ensuremath{\overset{\ensuremath{\mbox{\,}_\circ}}{g}}} 
\newcommand{\gflat}{ \ensuremath{g_E}} 
\newcommand{\nexp}{\ensuremath{{\frac2{n-2}}}} 
\newcommand\wh{\widehat}
\newcommand\wt{\widetilde}

\newcommand{\vol}{{\rm{ Vol}}}


\begin{document}

\title{Singular Yamabe problem for scalar flat metrics on the sphere}

\author{Aram L. Karakhanyan}
\address{School of Mathematics, The University of Edinburgh, Peter Tait Guthrie Road, EH9 3FD, Edinburgh, UK}
\email{ aram6k@gmail.com}
%
\maketitle

\begin{abstract}
Let $\Omega$ be a domain on the unit $n$-sphere $ \S^n$ and $\gstand$ the standard metric of $\S^n$, $n\ge 3$.
We show that  there exists a conformal metric $g$ with vanishing scalar curvature $R(g)=0$ 
such that $(\Omega, g)$  is complete  if and only if the Bessel capacity $\mathcal C_{\alpha, q}(\S^n\setminus \Omega)=0$, 
where $\alpha=1+\frac2n$ and $q=\frac n2$. 
Our analysis utilizes some well known properties of  capacity and Wolff potentials,  as well as  
a version of the Hopf-Rinow theorem for the divergent curves.
\end{abstract}
%

\section{Introduction}\label{sec-intro}
Let $\Omega$ be an open subset  of  the unit sphere $\S^n, n\ge 3$, and $\gstand$ the standard metric of 
$\S^n$ induced by the embedding $\S^n\hookrightarrow\R^{n+1}$. We want to characterize the open sets $\Omega$ with the following property: there exists a metric $g$, conformal to $\gstand$ such that $(\Omega, g)$ is complete and $g$ has vanishing scalar curvature. 
This question was studied by Schoen and Yau \cite{Sch-2}, \cite{SY-88}.
If we are given a compact Riemannian manifold $(M, g)$ then the problem of existence of a conformal deformation of 
$g$ into a complete metric $\bar g$ with constant scalar curvature is known as the Yamabe problem \cite{Yam}. Yamabe's original approach 
was to formulate the existence of $\bar g$ in terms of  a variational problem. Later contributions of Trudinger \cite{Neil}, Aubin \cite{Aubin} and Schoen \cite{Sch-1} helped to complete Yamabe's original approach.
  
One of the first results towards solving the Yamabe problem says that 
if there exists a complete $g$ conformal to $\gstand$ such that the scalar curvature $R(g)\ge 0$ 
then the Hausdorff dimension of $\p \Omega$ must be at most $\frac n2$,  \cite{SY-88}.
Under stronger
structural assumptions on $K:=\S^n\setminus\Omega$, one can show the converse statement. 
In particular, if $K=\S^n\setminus K$ is  a finite union of Lipschitz submanifolds of dimension $k\le (n-2)/2$ then 
there is a $g$, solving the Yamabe problem, such that $R(g)=0$,  see \cite{Del}, \cite{Kato-93}, \cite{McOwen-92}. 
See also \cite{Yam-16} for  the periodic setting with equator as singular set in $n\ge5$ sphere. 
Some discussion on this and related open problems is contained in \cite{McOwen-open}. For a survey of related recent results see \cite{Mar}, \cite{Denis} and the references therein.

Given a domain 
$\Omega\subset{\S}^n$. The scalar curvature of $R(g)$ of a metric $g$ after a conformal change $g=u^{\frac4{n-2}}\gstand$ is
\begin{equation}\label{Cozz}
R(g)=u^{-\frac{n+2}{n-2}}\left(-\frac{4(n-1)}{n-2}\Delta u +R(\gstand)u\right),
\end{equation}
where $\Delta u$ is the Laplace-Beltrami operator on the Riemannian manifold $(\S^n, \gstand)$. 
Suppose $g$ is scalar flat, that is, $R(g)=0$,  then 
from \eqref{Cozz} and the observation that $R(\gstand)=n(n-1)$ we arrive at the 
following problem:
\begin{eqnarray}
\label{PDE}
\frac{4(n-1)}{n-2}\Delta  u  -n(n-1)u =0&\ {\rm in}\ & \Omega,\nonumber\\
u>0 &\ {\rm in}\ &\Omega,\\
u^{4/(n-2)}\gstand \ {\rm is \ complete\ metric}&\ {\rm in}\ &\Omega.\nonumber
\end{eqnarray}
For the negative scalar case $R(g)=-1$ Labutin showed that a sufficient and necessary condition 
for the existence of $g$ is $\mathcal B_{2, \frac{n+2}4}(K)=0$, where $\mathcal B_{2, \frac{n+2}4}$ is 
the Bessel capacity for the Sobolev space
$W^{2, \frac{n+2}4}(\R^n)$, see \cite{Denis}.
His result can be seen as a way of measureing the thinness of $K$, which  comes from the classical potential theory, namely Wiener's test.
Recall that Wiener's theorem states that 
the Dirichlet problem 
\[
\left\{
\begin{array}{rcl}
\Delta w =0&{\rm in}&D, \\
w=f&{\rm on}&\partial D, 
\end{array}
\right.
\]
in a bounded domain 
$D\subset{\bf R}^n$,
$n\geq 3$,
is solvable for all
boundary data 
$f\in C(\partial D)$
if and only if
\begin{equation}\label{W-2}
\int_0^\delta
\frac{cap_2 (B(x, r)\setminus D)}{cap_2 (B(x, r))}
\,
\frac{dr}{r}
=+\infty \quad {\rm for \quad  any} \quad x\in\partial D.
\end{equation}
Here the upper limit $\delta>0$ is some fixed constant (e.g. one can take $\delta=1$), and $cap_2$ is the classical (electrostatic) capacity.

For $R(g)=0$, our goal  is to find a correct quantity measuring the thinness of $K$, similar to \eqref{W-2}, such that 
it gives a complete characterization of $K$ for which the existence of a solution to \eqref{PDE} follows. 
It turns out that Wolff's potential $\mathcal W^\mu_{\alpha, q}$ is the correct quantity in this sense
with $\alpha=1+\frac2n, q=\frac n2$,
 and hence the corresponding 
Bessel integral capacity  $\mathcal C_{1+\frac2n, \frac n2}$ is not the same as $\mathcal B_{\alpha, q}$.
These capacities, however, agree 
if $\alpha$ is an integer. 
In \cite{Denis} Labutin explicitly constructs a divergent curve (which are used to formulate a version of the Hopf-Rinow theorem for noncompact manifolds) to show that if $R(g)=-1$ and $K$ is not thin at some point then $g$ cannot be complete. This is the most 
technical part of \cite{Denis}.
However, for $R(g)=0$ the construction of a divergent curve simplifies if we   use 
the Schoen-Yau estimate of the Hausdorff dimension of $\partial \Omega$.

The aim of this work is to give a complete characterization of open set $\Omega$ without any assumption on the  structure of $K=\S^n\setminus \Omega$. 
In what follows $\capacity(\cdot) : =\mathcal C_{1+\frac2 n, \frac n2}(\cdot)$ stands for 
Bessel's capacity (see Section \ref{sec:2} for precise definitions). Our main result is  the following theorem: 

\begin{theorem}\label{th-A}
Let
$\Omega\subset \S^n, n\geq 3$,
be an open set and
$K={\S}^n \setminus \Omega$.
Then the following properties are equivalent:
\begin{itemize}
\item[(i)]
In 
$\Omega$
there exists a scalar flat complete metric conformal to $\gstand$.
\item[(ii)]
$\capacity(K)=0$.
\end{itemize}
\end{theorem}

The paper is organized as follows: 
Section \ref{sec:2}  contains some background material: first 
we use the steregraphic projection to reformulate the problem on $\R^n$. Then we introduce the Wolff potentials
and give a characterization of polar sets, i.e. sets of vanishing capacity,  see Proposition \ref{prop-Wolff}. We close the 
section by stating a non-standard version of the Hopf-Rinow theorem (for noncompact manifolds) formulated in terms of divergent curves, see Theorem \ref{thm:H-R}. This allows to link the finiteness of Wolff's potential with the completeness of the metric.

In Section \ref{sec:3} we show that the existence of a solution to scalar flat Yamabe problem implies that the 
capacity of $K$ is zero. A part of the argument is based on the representation of positive harmonic functions 
in terms of Martin kernels \cite{Armitage}, \cite{Helms}.
The crucial step in the proof is the estimate \eqref{u-est}. In the construction of the divergence curve 
we also use the Schoen-Yau estimate of the Hausdorff dimension  of $\p\Omega$.

Finally, in Section \ref{sec:4} we prove the implication $(ii)\Rightarrow(i)$ in Theorem \ref{th-A}.


\section{Background}\label{sec:2}
This section contains  some background results from conformal geometry and potential theory, see \cite{SchY} and \cite{Wolff}.
\subsection{Stereographic projection and reduction to $\R^n$}
Let $(M, g)$ be a Riemannian manifold of dimension $n\ge 3$.   
Let $R(g)$ be  the scalar curvature of the metric $g$
and $\Delta_g$ the Laplace-Beltrami operator. 
The operator 
\begin{equation}\label{eq-conf-lap}
\mathcal L_g=-4\frac{n-1}{n-2}\Delta_g u+R(g)
\end{equation}
 is called {\it conformal Laplacian}. 
 
It is well known that under conformal change of metric $\wh g=\phi^\frac 4{n-2}g, \phi\in C^\infty(M), \phi>0$,  we have 
 \begin{eqnarray}\label{eq-scal}
 R({\wh g})&=&\phi^{-\frac{n+2}{n-2}}\mathcal L_g\phi,\\
 \mathcal L_{\wh g} v&=&\phi^{-\frac{n+2}{n-2}}\mathcal L_g(\phi v).\label{eq-operator}
 \end{eqnarray}
Suppose $\wt M$ is  another manifold with metric $\wt g$, and let $f:M\to \wt M$ be a diffeomorphism changing the metric conformally. If $f^*\wt g=\wt g\circ f$ is the pull-back then 
 \begin{equation*}
 f^*\wt g=\phi^{\frac4{n-2}}g.
 \end{equation*}
Consequently, from \eqref{eq-scal} and \eqref{eq-operator} it follows that 
\begin{eqnarray}\label{eq-scal-1}
f^*(R_{\wt g})&=&\phi^{-\frac{n+2}{n-2}}\mathcal L_g\phi, \\\nonumber
f^*(\mathcal L_{\wt g}v)&=&\phi^{-\frac{n+2}{n-2}}\mathcal L_g(\phi f^*v).
\end{eqnarray}

Using  the stereographic projection $\sigma:\S^n\setminus\{N\}\to \R^n$, 
where $N$ is the north pole, we can rewrite the transformation equations on $\R^n$. Indeed, 
$\sigma$ is a diffeomorphism between  $(\S^n\setminus\{N\}, \gstand)$
and $(\R^n, \gflat)$,  because
\begin{eqnarray*}
(\sigma^{-1})^*\gstand &=&\left(\frac2{1+|x|^2}\right)^2 \gflat\\
&=& U^{\frac4{n-2}} \gflat, 
\end{eqnarray*}
where
\begin{eqnarray*}
U(x):=\left(\frac2{1+|x|^2}\right)^{\frac{n-2}2}\quad x\in \R^n.
\end{eqnarray*}
Since we consider the scalar flat case, i.e. $R({\wt g})=0$, then \eqref{eq-scal-1} yields
\begin{equation*}
{\mathcal L}_{\gstand} v=0, \quad v>0\  \ \hbox{in} \ \  \Omega.
\end{equation*}
Introduce the function 
\begin{eqnarray*}
u(x)&=&U(x)(\sigma^{-1})^*v(x)\\\nonumber
&=&U(x)v(\sigma^{-1}x), \quad x\in \R^n.
\end{eqnarray*}
Then from \eqref{eq-operator} we obtain 
\begin{equation*}
{\mathcal L}_{\gflat} v=0, \quad v>0\ \ \hbox{in}\ \ \sigma(\Omega)\subset \R^n.
\end{equation*}
Since $R({\gflat})=0$ then we get 
\begin{equation*}
\Delta_{\gflat }v=0, \quad v>0\ \ \hbox{in}\ \ \sigma(\Omega)\subset \R^n.
\end{equation*}

\subsection{Characterization of capacity}
For  $\alpha>0, 1<q\le \frac n{\alpha}$ we define the Bessel capacity of $E\subset \R^n$ as follows 
\[
\mathcal C_{\alpha, q}(E)=\inf\left\{\int_{\R^n}\psi^q ;\ \psi\ge 0, G_\alpha*\psi\ge 1\ \text{for all}\ x\in E\right\}, 
\]
where $G_\alpha$ is the Bessel kernel, best defined as the inverse Fourier transform of $(1+|\xi|^2)^{-\frac\alpha2}$, see \cite{Wolff}.
For given Radon measure $\mu$  the Wolff potentials are defined as follows 
\begin{equation}\label{eq:W-def}
\mathcal W_{\alpha, q}^\mu(x)=\int_0^1\left(\frac{\mu(B(x, \delta))}{\delta^{n-\alpha q}}\right)^{p-1}\frac{d\delta}{\delta},
\end{equation}
where $p+q=pq$.

An important fact is that $\mathcal W^\mu_{\alpha, q}$ bounds the nonlinear potential
 $\mathcal V_{\alpha, q}^\mu(x):=G_\alpha*(G_\alpha*\mu)^{p-1}$ from below (see \cite{Wolff}, page 164), i.e. there is a constant $A>0$ such that 
 \begin{equation}\label{eq:W-V}
\mathcal W_{\alpha, q}^\mu(x)\le A \mathcal V_{\alpha, q}^\mu(x).
\end{equation}

In what follows we take $\alpha=1+\frac2 n, q=\frac n2$ and denote 
the resulted capacity by $\capacity(\cdot)=\mathcal C_{1+\frac2n, \frac n2}(\cdot)$. Then for this choice of parameters the Wolff potential, which we denote by $\mathcal W^{\mu}$ for short, takes the following form 
\begin{equation*}
\mathcal W^\mu_{1+\frac2n, \frac n2}(x):= \mathcal W^\mu (x)=
\int_0^1
\left(
\frac{\mu\left( B(x,r)\right)}{r^{\frac{n-2}2}}
\right)^\nexp
\frac{dr}{r}.
\end{equation*}
The corresponding capacity on $\S^n$ is defined accordingly. By abuse of notation,  we continue to use the same notation $\capacity(\cdot)$.
The following characterization of vanishing capacity compacts is hard to find in the literature, so we give a proof for the reader's convenience. 
\begin{prop}
\label{prop-Wolff}
Let $K\subset\R^n$ be a compact set then $\capacity(K)=0$
if and only if  there exists a Radon measure $\mu$, $\|\mu\|=1$, such that $\supp \mu\subset K$ and
\begin{equation}
\label{Wolff_infty}
\mathcal W^\mu( x)=+\infty
\quad
{\rm for \quad all \quad }
x\in K.
\end{equation}
\end{prop}

\begin{proof}
Suppose there is a measure $\mu$ such that $\supp \mu\subset K, \|\mu\|=1$ and 
$\mathcal W(\mu, x)=+\infty$ for all $x\in K$. From \eqref{eq:W-V}
it follows that $V_{\alpha, q}^\mu(x)=\infty$ for all $x\in K$. Note that 
$f=(G_\alpha*\mu)^{p-1}\in L^q(\R^n)$, see \cite{Wolff} page 163. By Theorem 1.4 \cite{Reshetnyak} $\capacity(E)=0$ if and only if there is a nonnegative $f\in L^q$ such that 
$G_\alpha*f=\infty$ for all $x\in E$. Since by definition 
$G_\alpha*f=\mathcal V_{\alpha, q}^\mu$, and thanks to \eqref{eq:W-V}, it follows that $\capacity(E)=0$. 

\smallskip 

Now we prove the converse statement. Suppose $\capacity(E)=0$. By Proposition 4 \cite{Wolff} we have 
\[
0=\inf_\mu \left\{ \int \mathcal W^\mu d\mu : \ \mathcal W^\mu\ge 1 \ \text{on}\ E\right\}.
\]
Let $\mu_j$ be a minimizing sequence such that $\int \mathcal W^{\mu_j}d\mu_j\to 0$
and $\|\mu_j\|>0$. Such sequence exists because $\mathcal W^\mu\ge 1$ on $E$, see \eqref{eq:W-def}. 
This inequality also implies that the restricted measures 
$\|\mu_j|_E\|\le \int \mathcal W^{\mu_j}d\mu_j\to 0$ have vanishing mass in the limit.
Introduce the unit mass measures $\widehat \mu_j=\mu_j|_{E}/\mu_j(E)$. 
Then from the scaling property of $W^\mu$ we get that 
\[
W^{\widehat \mu_j}_{\alpha, q}\ge \frac1{(\mu_j(E))^{p-1}}\to \infty\quad \text {on}\ E.
\]
Then the existence follows from a customary compactness argument for $\widehat \mu_j$ and the semicontinuity of the potential $\mathcal W^{\mu}$. 
\end{proof}

\subsection{Divergent curves and Hopf-Rinow theorem}
In order to characterize the completeness of the metric by Wolff's potential we state 
a version of the Hopf-Rinow theorem formulated in terms of divergent curves. 
\begin{definition}
A divergent curve in a Riemannian manifold $M$ is a differentiable mapping 
$c: [0,T)\to M$ such that for any compact set
$K\subset M$ there exists $t_0 \in(0, T)$ with $c(t) \not \in  K$ for all $t > t_0$.
\end{definition}
In other words, if $c$ is a divergent curve then it "escapes" every compact set in $M$. 
Define the
length of a divergent curve by
\begin{equation}\label{eq:curve-length}
L_g(c)=\lim_{t\to T}\int_0^t|c'(\tau)|d\tau.
\end{equation}
Then we have the following version of the Hopf-Rinow theorem.

\begin{theorem}\label{thm:H-R}
Let $M$ be a noncompact Riemannian manifold. Then $M$ is complete if and only if the length of any
divergent curve is unbounded.
\end{theorem}

\begin{proof}
Suppose $M$ is complete then the classical Hopf-Rinow theorem holds. Consequently,   the closed ball 
$\overline{B(0, N)}$ is compact for every $N=1, 2, \dots$. 
Let $c:[0, T)\to  M$ be a divergent curve, then for every $N$ there is $t_N\in (0, T)$ such that $c(t_N)\not \in \overline{B(0, N})$. Therefore, from \eqref{eq:curve-length}
we get that 
\[
L_g(c|_{[0, t_N)})\ge N\to \infty,
\]
and, hence $c$ is a divergent curve. 

\smallskip

Now suppose that every divergent curve has infinite length. We want to show that then this implies that $M$ is complete. If $M$ is not complete then 
there is a geodesic $c : [0, T)\to M$ such that  $c$ cannot be extended further than $T$. 
Observe that $c$ has constant speed since
\[
c'g(c', c')=2g(\nabla_{c'}c', c')=0, 
\] 
where $\nabla$ is the Riemannian connection on $M$. Therefore $L_g(c)\le \theta T$ for some constant $\theta$. 
To finish the proof it is enough to show that $c$ is a divergent curve, hence $c$ cannot have finite length.
Thus,  suppose that $c$ is contained in some compact $K$. Let us take $t_k\in (0, T)$ so that $p_k:=c(t_k)\in K$ and $\lim_{k\to \infty}t_k=T$. We can extract a subsequence so that 
$c(t_{k_m})\to p$ for some $p\in K\subset M$. Let $W$ be a totally normal neighborhood of $p$, that is, $W$ is a normal neighborhood of all of its points. The existence of $W$ is well known,  see \cite{doCarmo}, page 72. 
Consequently, if $\epsilon>0$ is small so that $c(T-\epsilon)\in W$, then the geodesic joining 
$c(T-\epsilon)$ and $p$ can be continued further than $p$, which is a contradiction.
Thus $c$ escapes any compact and hence it is a divergent curve.
\end{proof}


\section{Proof of $(i)\Rightarrow(ii)$: existence of $u$ implies $cap(K)=0$}\label{sec:3}
We can use the stereographic projection and, thanks to the conformal homothety on $\R^n$, without loss of 
generality assume that the north pole $N\in \Omega$ such that $\sigma(K)\subset B(0, 1/2)$, 
and there is $u:\R^n\setminus \sigma(K)\to \R$ such that $g=u^{\frac 4{n-2}}\gstand$ is complete,  where  
\begin{equation*}
\Delta u =0, \quad \text{in}\ \  \sigma(\Omega)=\{u>0\}.
\end{equation*}
We claim that there exists a Radon measure $\mu$, 
with  $\supp \mu \subset K$, such that the following representation of $u$ is true
\begin{equation}
\label{Martin_representation}
u(x)=\int_{\R^n} k(x,y)d\mu(y)\quad \forall x\in B(0,3)\setminus K,
\end{equation}
where $k(x, y)$ is the Martin kernel (see \cite{Armitage} Theorem 8.4.1 or Chapter 12 \cite{Helms} p 251) and $k$ is locally integrable in 
$\R^n\times \R^n$. Moreover, there exists a  universal  constant $C_M$ such that 
\begin{equation}\label{anhav}
0\le k(x,y)\le \frac{C_M}{|x-y|^{n-2}}, 
\end{equation}
see \cite{Helms}, chapter 12.

Suppose $\capacity(K)>0.$ By Proposition \ref{prop-Wolff}
the Wolff potential of 
$\mu$
must be  finite at some point $x_0\in K$.
Without loss of generality we assume that $x_0=0$ and $\mu$ is a probability measure such that 
\begin{equation*}
\mathcal W^\mu(0)<+\infty, \quad 0\in K.
\end{equation*}

We first establish  a useful estimate.  
\begin{lemma}
Let $u>0$ be as above and $\supp \mu\subset B(0, 1/2)$. Then 
there is a constant $C>0$ such that 
\begin{equation}\label{u-est}
\int_{B(0, 1)}
\frac{\left(u(x)\right)^\nexp}{|x|^{n-1}}
\,
dx \le C \mathcal W^\mu( 0) .
\end{equation}

\end{lemma}
\begin{proof}
Denote $D_m=B(0, \rho_{m})\setminus B(0, \rho_{m+1}), \rho_m=2^{-m}$. We have 
\begin{eqnarray}\label{blya-11}
\int_{B(0,2)} \frac{u(x)^\nexp}{|x|^{n-1}} \,dx
&=&
\sum_{m=0}^\infty\int_{D_m}\frac{u(x)^\nexp}{|x|^{n-1}}\,dx\nonumber\\
&&
+
\int_{B(0,2)\setminus B(0,1)}\frac{u(x)^\nexp}{|x|^{n-1}}\,dx\nonumber
\\
&=&I_1+I_2.
\end{eqnarray}
Since $u=k*\mu$ then from  (\ref{anhav}) we see that 
\begin{equation*}
u\le C \quad{\rm in}\quad B(0, 2)\setminus B(0,1).
\end{equation*}
Hence
\begin{equation} \label{II}
I_2 \le C.
\end{equation}
As for  $I_1$ in (\ref{blya-11}), we have 
\begin{eqnarray*}
&&
\int _{D_m}\dfrac {1}{\left| x\right| ^{n-1}}\left( \int _{B(0, {2})}\dfrac {d\mu \left( y\right) }{\left| x-y\right| ^{n-2}}\right) ^{\frac {2}{n-2}}dx=\\
&&=\int _{D_m}\dfrac {1}{\left| x\right| ^{n-1}}
\left[\int\limits _{B\left( 0,\rho _{m+2}\right)}\dfrac {d\mu \left( y\right) }{\left| x-y\right| ^{n-2}}
+
\int \limits_{B\left( 0,\rho _{m-2}\right)\setminus B(0, \rho_{m+2})}\dfrac {d\mu \left( y\right) }{\left| x-y\right| ^{n-2}}
+ 
\int\limits _{B(0, 1)\setminus B\left( 0,\rho _{m-2}\right)}\dfrac {d\mu \left( y\right) }{\left| x-y\right| ^{n-2}}
\right]^{\frac 2{n-2}}dx.
\end{eqnarray*}
For $x\in D_{m}$ we have 
\[\int _{B\left( 0,\rho _{m+2}\right) }\dfrac {d\mu \left( y\right) }{\left| x-y\right| ^{n-2}}\leq \dfrac {1}{\rho ^{n-2}_{m+2}}\mu \left( B\left( 0,\rho_{m+1}\right) \right), \]
and
\[
\int _{B(0, 1)\backslash B\left( 0,\rho _{m-2}\right) }\dfrac {d\mu \left( y\right) }{\left| x-y\right| ^{n-2}}=\sum ^{m-3}_{k=0}\int _{D_k}\dfrac {d\mu \left( y\right) }{\left| x-y\right| ^{n-2}}.
\]
Noting that 
\[\dfrac {1}{2^{k+1}}-\dfrac {1}{2^{m}}=\dfrac {1}{2^{k+1}}\left( 1-\dfrac {1}{2^{m-k-1}}\right) \geq \dfrac {3}{4}\dfrac {1}{2^{k+1}},\]
we get 
\[\sum ^{m-3}_{k=0}\int _{D_k}\dfrac {d\mu \left( y\right) }{\left| x-y\right| ^{n-2}}\leq \left( \dfrac {8}{3}\right) ^{n-2}\sum ^{m-3}_{k=0}\dfrac {\mu \left( B\left( 0,\rho_{k}\right) \right) }{\rho^{n-2}_{k}}.\]
Combining, we obtain the estimate  
\begin{eqnarray}\nonumber
I_1
&\le& 
C \left( n\right) \sum ^{\infty }_{m=0}
\Bigg\{
\int \limits_{D_m}\dfrac {1}{\left| x\right| ^{n-1}}\left( \sum ^{m-3}_{k=0}\dfrac {\mu \left( B\left( 0,\rho_{k}\right) \right) }{\rho^{n-2}_{k}}\right) ^{\frac {2}{n-2}}
+
\left(\frac{\mu \left( B\left( 0,\rho_{m+1}\right) \right)}{\rho ^{n-2}_{m+2}} \right)^{\frac2{n-2}}\\\nonumber
&&+
   \left(
       \int \limits_{B\left( 0,\rho _{m-2}\right)\setminus B(0, \rho_{m+2})}\dfrac {d\mu \left( y\right) }{\left| x-y\right| ^{n-2}}
   \right)^{\frac 2{n-2}}
\Bigg\}\\\nonumber
&\le&
C \left( n\right) \sum ^{\infty }_{m=0}
\left\{\rho_m\left( \sum ^{m}_{k=0}\dfrac {\mu \left( B\left( 0,\rho_{k}\right) \right) }{\rho^{n-2}_{k}}\right) ^{\frac {2}{n-2}}
+\frac1{\rho_m^{n-1}} \int_{D_m}\left(\int \limits_{B\left( 0,\rho _{m-2}\right)\setminus B(0, \rho_{m+2})}\dfrac {d\mu \left( y\right) }{\left| x-y\right| ^{n-2}}\right)^{\frac 2{n-2}}dx
\right\}\\\nonumber
&=&
C( I_3+I_4).
\end{eqnarray}

For $n=3$ we take a sequence of smooth functions $f_i$ weakly converging to $\mu$ in 
$\wt {D_m}: =B\left( 0,\rho _{m-2}\right)\setminus B(0, \rho_{m+2})$ (see Lemma 0.2 \cite{Landkof}) then applying lemma 7.12 from \cite{GT} to 

\[V_sf_i\left( x\right) =\int _{\wt D_m}\left| x-y\right| ^{n\left( s-1\right) }f_i(y)dy\]
with $q=2,p=1, \delta =1-\dfrac {1}{q}=\dfrac {1}{2}$, and $s=\delta +\dfrac {1}{6}=\frac23$,  we get 

\[\int _{\wt D_m}\left| V_{2/3}f_{i}\right| ^{2}\leq C\left( \vol_{\gflat} (\wt D_{m})\right) ^{\frac {1}{3}}\left( \int _{\wt D_m}f_{i}\right) ^{2}.\]
After letting $i\to \infty$ this yields 
\[I_4\le \sum ^{\infty }_{m=0}\dfrac {1}{\rho _{m}}\left( \mu \left( B\left( 0,\rho _{m}\right) \right) \right) ^{2}\leq C\int ^{1}_{0}\dfrac {\left( \mu \left( B\left( 0,t\right) \right) \right) ^{2}}{t^{2}}dt.\]
Moreover, denoting $m(t)=\mu(B(0, t))$ and using integration by parts together with  Cauchy-Schwarz inequality we get 
\begin{eqnarray*}
\int ^{1}_{0}\left( \int ^{1}_{t}\dfrac {m\left( \tau \right) }{\tau ^{2}}d\tau \right) ^{2}dt
&=&
2\int ^{1}_{0}\dfrac {m\left( t\right) }{t}\left( \int ^{1}_{t}\dfrac {m\left( \tau \right) }{\tau ^{2}}d\tau \right) dt\\
&\le& 
2\left[\int ^{1}_{0}\left(\dfrac {m\left( t\right) }{t}\right)^2dt\int_0^1\left(\int ^{1}_{t}\dfrac {m\left( \tau \right) }{\tau ^{2}}d\tau \right)^2 dt\right]^{\frac12},
\end{eqnarray*}
implying that 
\[
\int ^{1}_{0}\left( \int ^{1}_{t}\dfrac {m\left( \tau \right) }{\tau ^{2}}d\tau \right)^{2}dt
\le
4\int ^{1}_{0}\left( \dfrac {m\left( t\right) }{t}\right)^2dt.
\]
Hence 
\[
I_3\le C\int ^{1}_{0}\left( \int ^{1}_{t}\dfrac {m\left( \tau \right) }{\tau ^{2}}d\tau \right)^{2}dt
\le
4C\int ^{1}_{0}\left( \dfrac {m\left( t\right) }{t}\right)^2dt=4C\mathcal W^\mu(0).
\]

If $n=4$  we have
\[
\int _{\wt {D_m}}d\mu \left( y\right) \int _{D_m}\dfrac {dx}{\left| x-y\right| ^{2}}\leq \mu \left( B\left( 0,\rho_{m}\right) \right) \rho ^{2}_{m},
\]
and then from Fubini's theorem we get, as above, the bound 
\[
I_4\le\sum ^{\infty }_{m=0}\dfrac {1}{\rho _{m}}\mu \left( B\left( 0,\rho _{m}\right) \right) \leq C\int ^{1}_{0}\dfrac {\mu \left( B\left( 0,t\right) \right)}{t^{2}}dt.\
\]
The estimate for $I_3$ follows from integration by parts.

Finally, let us consider the case $n\ge 5$. We have 
\begin{eqnarray*}
\int _{D_m}\left( \int_{\wt D_m} \dfrac {d\mu \left( y\right) }{\left| x-y\right| ^{n-2}}\right) ^{\frac {2}{n-2}}
&\leq& 
\left( \int _{D_m}\int_{\wt D_m} \dfrac {d\mu }{\left| x-y\right| ^{n-2}}\right)^\frac {2}{n-2}\left(\vol_{\gflat}( D_{m})\right) ^{1-\frac {2}{n-2}}\\
&\leq& 
C\left( \rho^{2}_{m}\mu(B(0, \rho_m)) \right) ^{\frac {2}{n-2}}\rho ^{n\left( 1-\frac {2}{n-2}\right) }_{m}\\
&=&
C\left(\mu(B(0, \rho_m)) \right) ^{\frac {2}{n-2}}\rho ^{n-2}_{m}.
\end{eqnarray*}
Thus 
\[
I_4\le  \sum ^{\infty }_{m=0}\dfrac {1}{\rho _{m}}\left( \mu \left( B\left( 0,\rho _{m}\right) \right) \right) ^{\frac2{n-2}}\leq C\int ^{1}_{0}\dfrac {\left( \mu \left( B\left( 0,t\right) \right) \right) ^{\frac2{n-2}}}{t^{2}}dt.
\]
As for $I_3$, one can easily see that 
\[
\left( \sum ^{m}_{k=0}\dfrac {\mu \left( B\left( 0,\rho_{k}\right) \right) }{\rho^{n-2}_{k}}\right) ^{\frac {2}{n-2}}
\le C\sum ^{m}_{k=0}\left( \dfrac {\mu \left( B\left( 0,\rho_{k}\right) \right) }{\rho^{n-2}_{k}}\right) ^{\frac {2}{n-2}},
\]
and consequently after integration by parts we get 
\begin{eqnarray*}
I_3\le C \int ^{1}_{0} \int ^{1}_{t}\left(\dfrac {m\left( \tau \right) }{\tau ^{n-2}} \right) ^{\frac2{n-2}}\frac{d\tau}\tau{dt}
\le C \int ^{1}_{0} \left(\dfrac {m\left( t \right) }{t ^{n-2}} \right) ^{\frac2{n-2}}{dt}.
\end{eqnarray*}
The proof of lemma is complete.
\end{proof}

{\it Proof of (i)$\Rightarrow$(ii)
in Theorem~\ref{th-A}.}
We claim that there exists a smooth curve $c:[0,1]\to B(0, 2)\setminus K$ such that 
\begin{eqnarray}
\label{fin}
&&
L_g(c)=\int_0^1(u(c(t)))^\nexp\left|{c'}(t)\right| dt<+\infty, \nonumber \\\label{bama}
&&
{\rm and\ }\gamma(t)\to 0\ {\rm as}\ t\to 1.
\end{eqnarray}
Observe that 
(\ref{fin})
is impossible if 
$u^{4/(n-2)}\gflat$ is complete, thanks to Theorem \ref{thm:H-R}.
Hence to finish the proof we have to establish (\ref{fin}).
So it's left to show the existence of such curve.
Take 
$\omega\in\partial B$
and let  $\ell(\omega)$ be the interval 
\begin{equation*}
\ell(\omega)=
\left\{x\in {\bf R}^n:\quad x=s\omega, \quad0<s\leq 1 \right\},
\end{equation*}
where $\pi : x\to  \frac x{|x|}$ is the projection on $\partial B$.
Put 
\begin{equation*}
\Xi=
\partial B \setminus \pi(K \setminus\{0\}), 
\end{equation*}
and observe that 
\begin{equation*}
\pi^{-1}( \Xi ) \subset B\setminus K.
\end{equation*}
We claim that $\H^{n-1}(\Xi)>0$, otherwise this means that $\H^{n-1}(\p \Omega)>0$
But this will be in contradiction with the Schoen-Yau estimate of the 
Hausdorff dimension of $\p\Omega$ which can be at most $\frac n2$, see \cite{SY-88}, Theorem 2.7.  

Switching to polar coordinates 
$(r,\omega)$, $r>0$, $\omega\in\partial B$, we get  from \eqref{u-est} that
\begin{eqnarray*}
+\infty
&>&
\int_{\pi^{-1}( \Xi ) } u(x)^\nexp\,\frac{1}{|x|^{n-1}}\, dx\\
&=&
\int_{\Xi}\int_0^1 u(x(r,\omega))^\nexp\,\frac{1}{r^{n-1}}\,r^{n-1}\, dr\, d\H^{n-1}(\omega)\\
&=&
\int_{\Xi}\left(\int_{\ell(\omega)}u^\nexp\, ds\right)\, d\H^{n-1}(\omega).
\end{eqnarray*}
Consequently
\begin{equation*}
\int_{\ell(\omega_0)}u^\nexp
\, ds
<+\infty
\quad{\rm for\quad some\quad}
\omega_0 \in \Xi.
\end{equation*}
By our definitions
\begin{equation*}
\ell(\omega_0)\cap K =\emptyset,
\end{equation*}
and we conclude that  \eqref{bama} holds for the curve $\gamma=\ell(\omega_0)$. This finishes the proof. 


\section{Proof of $(ii)\Rightarrow (i)$: $\capacity(K)=0$ implies existence of metric}\label{sec:4}

{\it Proof of (ii)$\Rightarrow$(i)
in Theorem~\ref{th-A}.}
In what follows we assume, without loss of generality, that the north pole $N\in\Omega$. 
Since $\sigma(K)\subset\R^n$ is the image of $K$ under stereographic projection then it is compact
such that
\begin{equation*}
\capacity(\sigma(K))=0.
\end{equation*}
From Proposition \ref{prop-Wolff}
it follows that there is a probability measure $\mu$,
such that $\supp\mu\subset \sigma(K)$ and
\begin{equation}\label{Wolff-infty}
\mathcal W^\mu( x)=+\infty\quad {\rm for \quad all \quad }x\in \sigma(K).
\end{equation}
The convolution 
\begin{equation*}
u(x) = \int\frac{d\mu(y)}{|x-y|^{n-2}}
\end{equation*}
solves 
\begin{equation*}
\Delta  u   =0, \quad  u>0\quad{\rm in}\quad \R^n\setminus \sigma(K).
\end{equation*}
To finish the proof we have to show that 
\begin{eqnarray}\label{blya}
\sigma^* 
\left( u^{\frac4{n-2}}\gflat \right) \ {\rm is \ a \ complete\ metric}\ {\rm in}\ \Omega.
\end{eqnarray}
Since $u$ is harmonic in $\{u>0\}$, and $u(\sigma(N))>0$ then $\sigma^*(u^{4/(n-2)})\gflat$ gives a metric on $\S^n$ which is smooth at $N$.
To check \eqref{blya},  we use a version of the Hopf-Rinow 
theorem formulated in terms of divergent curves, see Theorem \ref{thm:H-R}. 
Let us take a divergent curve $c$  in $\Omega$, and denote $\wt c : [0,+\infty)\to \R^n\setminus \sigma(K)$  its stereographic projection.
Clearly $c$ is divergent curve in 
$\R^n\setminus \sigma(K)$. Since by assumption $N\in \Omega$ then $\wt c$ is contained in some ball in $\R^n$.
Recall the arc length formula 
\begin{equation}\label{length_infty}
L_{g}(c)=\int_0^\infty \sqrt{g(c'(t), c'(t))}dt=
\int_0^\infty
u(\wt c(t))^{\frac 2{n-2}}\,\left|{\wt c\,{}'}(t)\right|\,dt,
\end{equation}
where $g=u^{\frac4{n-2}}\gflat.$

By assumption  $c$ (and hence $\wt c$)  is a  divergent curve,  therefore there exists $x_0\in K$ such that
\begin{equation}\label{blya-2}
\dist_{\gstand}(x_0, c(t_k))
\to 0
\quad
{\rm for\quad a \quad sequence}
\quad
\{t_k\},
\quad
k\to+\infty.
\end{equation}
For
$m\in \mathbb N$, we let
$\gamma_m=D_m\cap \wt c$,  
where 
\begin{equation*}
D_m=\left\{x\in {\R}^n:\frac1{2^m}<|x-\wt x_0|<\frac1{2^{m-1}}\right\}.
\end{equation*}

If $m\geq m_0$, for sufficiently large $m_0$,  it follows from the smoothness of
$c$ that $\gamma_m$ is at most a countable union of open smooth curves. 
Moreover, from \eqref{blya-2} we see that  $\gamma_m\ne\emptyset$, and 
\begin{equation*}
L_{\gflat}(\gamma_m)\geq \frac1{2^{m-1}}-\frac1{2^{m}}=\frac1{2^{m}}
\quad{\rm for \quad all } \quad m \geq m_0.
\end{equation*}
For $y\in B(\wt x_0, 2^{-(k+2)})$ and $x\in D_k$ we have $|x-y|\le \frac1{2^k}+\frac1{2^{k+2}}=\frac5{2^{k+2}}$. Therefore
\begin{eqnarray}\label{inf-blya}\nonumber
u(x)
&
\geq
&
\int_{B(\wt x_0, \rho_{k+2})}\frac{d\mu(y)}{|x-y|^{n-2}}
\\
&
\ge
&
\frac1{5^{n-2}}\frac{\mu\left( B(\wt x_0, \rho_{k+2})\right)}{\rho_{k+2}^{n-2}}
\quad {\rm for\ all}\ x\in D_k,
\end{eqnarray}
where we set $\rho_i=2^{-i}.$
Let $I_k\subset (0,+\infty)$ denote the open set such that
\begin{equation*}
\wt c : I_k\to \sigma(\Omega)\cap D_k.
\end{equation*}
Then we derive that
\begin{eqnarray*}
L_g(c)
&=&\int_0^\infty
u(\wt c(t))^{\frac 2{n-2}}\,\left|{\wt c\,{}'}(t)\right|\,dt\\
&=&\sum_{k=0}^\infty \int_{I_k}
u(\wt c(t))^{\frac 2{n-2}}\,\left|{\wt c\,{}'}(t)\right|\,dt\\
&\ge&
\sum_{k=0}^\infty
\left(
\inf_{D_k} u \right)^{\frac2{n-2}}L_{\gflat}(\gamma_k)\\
&\ge&
\frac2{5^{n-2}}\sum_{k=0}^\infty
\left(\frac{\mu\left( B(\wt x_0, \rho_{k+2})\right)}{\rho_{k+2}^{n-2}}\right)^{\frac2{n-2}}
\rho_k \qquad {\text{after using} \ } \eqref{inf-blya}.
\\
\end{eqnarray*}
Recalling the definition of $\mathcal W^\mu$ we see that 
\begin{eqnarray*}
\mathcal W^\mu(x)&=&\int_0^1 \left(
\frac{\mu\left( B(x,r)\right)}{r^{n-2}}\right)^\nexp dr
=\sum_{k=0}^\infty\int_{2^{-(k+1)}}^{2^{-k}} \left(
\frac{\mu\left( B(x,r)\right)}{r^{n-2}}
\right)^\nexp dr\\
&\le& 4\sum_{k=0}^\infty \left(
\frac{\mu\left( B(x, \rho_k)\right)}{\rho_k^{n-2}}
\right)^\nexp \rho_k. \\
\end{eqnarray*}
Comparing the inequalities for $L_g(c)$ and $\mathcal W^\mu$ and recalling (\ref{Wolff-infty})
\begin{equation*}
\mathcal W^\mu(\wt x_0)=+\infty, 
\end{equation*}
we obtain that $\widetilde c$ (and hence $c$) has infinite length, and thus Theorem \ref{thm:H-R} implies that 
\eqref{blya} is true.
\qed

\end{document}